\def \eps {\varepsilon}

\documentclass[12pt]{amsart}

\usepackage{fullpage}
\usepackage{amsfonts}
\usepackage{bbding}
\usepackage[dvips]{graphicx}
\usepackage{amsfonts,amssymb,amsmath,amsthm,amscd}
\usepackage{mathrsfs}
\usepackage{xcolor}
\usepackage{empheq}
\usepackage{adforn}
\usepackage{cancel}
\usepackage{xr-hyper}
\usepackage{xr}
\usepackage{mdframed}
\usepackage{mathdots}
\usepackage{enumerate}
\usepackage{lmodern}
\usepackage{mdframed}
\usepackage{stmaryrd}
\usepackage{blindtext}
\usepackage[all, knot]{xy}
\usepackage{leftidx}
\usepackage{accents}

\definecolor{slightblue}{rgb}{.8, .8, 1}
\definecolor{hair}{RGB}{100,225,190}
\definecolor{ruby}{RGB}{220,50,120}
\definecolor{grass}{RGB}{150,220,110}

\usepackage[colorlinks=true, pdfstartview=FitP, linkcolor=hair!80!black, citecolor=hair!80!black, urlcolor=hair!80!black]{hyperref}

\usepackage[T1]{fontenc}

\newtheorem{theorem}{Theorem}
\newtheorem{proposition}[theorem]{Proposition}
\newtheorem{lemma}[theorem]{Lemma} 
\newtheorem{corollary}[theorem]{Corollary}

\theoremstyle{definition}

\theoremstyle{remark}  \numberwithin{equation}{section}
\numberwithin{figure}{section}

\newcommand{\Pb}{\mathbb{P}}

\newcommand{\Rb}{\mathbb{R}}

\newcommand{\Ub}{\mathbb{U}}

\newcommand{\Cc}{\mathcal{C}}
\newcommand{\Dc}{\mathcal{D}}
\newcommand{\Ec}{\mathcal{E}}
\newcommand{\Fc}{\mathcal{F}}

\newcommand{\Nc}{\mathcal{N}}

\newcommand{\Sc}{\mathcal{S}}
\newcommand{\Tc}{\mathcal{T}}

\usepackage{sidecap,subcaption}
\usepackage[section]{placeins}
\usepackage{wrapfig}
\usepackage{lpic}

\sidecaptionvpos{figure}{c}

\begin{document}

\title {The law of a point process of Brownian excursions\\ in a domain is determined by the law of its trace}

\author{Wei Qian}  
\author{Wendelin Werner}
\address {
 Center for Mathematical Sciences, University of Cambridge,
Wilberforce Rd., Cambridge CB3 0WB, United Kingdom}
\email {wq214@cam.ac.uk}
\address {
Department of Mathematics,
ETH Z\"urich, R\"amistr. 101,
8092 Z\"urich, Switzerland}

\email{wendelin.werner@math.ethz.ch}
\date{}

\begin {abstract}
We show the result that is stated in the title of the paper, which has consequences about decomposition of  Brownian loop-soup clusters in two dimensions. 
\end {abstract}

\maketitle

\section{Introduction}

\subsection*{Main result of the present paper and strategy of proof}

When $x$ and $y$ are two distinct boundary points of the unit disk $\Ub$,
 let us denote by $P_{x,y}$ the natural probability measure on Brownian excursions from $x$ to $y$ in $\Ub$ (that can be for instance defined as the limit when $z \to x$ of the law of 
 Brownian motion started from $z \in \Ub$ and conditioned to exit $\Ub$ at $y$, see \cite{MR2045953, MR1796962}). Using their conformal invariance properties, one can also define such 
 Brownian excursions in any simply connected planar domain, and it is 
also easy to generalize the definition to multiply connected domains. 

It is easy to see  that the image of $P_{x,y}$ under time-reversal is $P_{y,x}$. 
This leads to the definition of an {\emph{unoriented}} excursion which is obtained from 
an oriented excursion by forgetting its orientation (i.e., we say that an excursion and its time-reversal represent the same 
unoriented excursion). In the sequel, we will denote an ordered pair of distinct endpoints by $(x,y)$  and the
corresponding unordered pair by the set $\{ x,y\}$.

Suppose now that we are given a point process $\Cc = (\{ x_i,y_i\})_{i \in I}$ of unordered pairs of distinct points on $\partial \Ub$ 
(let us stress that this point process is not necessarily a Poisson point process). 
Then, conditionally on this point process, one can independently sample for each $i$ an unoriented Brownian excursion $E_i$ using the measure $P_{x_i, y_i}$. 
This gives rise to the point process of unoriented excursions $\Ec=(E_i)_{i \in I}$. 

Throughout this paper, we will assume that the law of $\Cc$ is chosen so that almost surely, 
$$ \sum_i | x_i -y_i|^2 < \infty$$
(where here and in the rest of the paper, $|\cdot|$ denotes the Euclidean norm in the complex plane). 
It is immediate to check via Borel-Cantelli's lemma that this property of $\Cc$ is equivalent to the property that 
the corresponding point process of excursions $\Ec$ is almost surely {\em locally finite}, i.e., that for every positive $\eps$, only finitely many of the excursions
in $\Ec$ do have a diameter greater than $\eps$.
Let us stress that this does not exclude the possibility (and this is a case that motivates 
the present work) that $\Ec$ has infinitely many excursions and that the union of all pairs in $\Cc$ is almost surely dense on $\partial \Ub$.

One can then define the trace ${\mathcal T}(   \Ec)$ of $\Ec$ as the closure of the 
union of the traces of all the excursions $E_i$. This is a random compact set in the closed unit disk.
Note that because of the local finiteness, knowing this trace is in fact the same as knowing the cumulative occupation time measure
(that associates to each open set $A$ the cumulated time spent in $A$ by all the excursions $E_i$, see \cite{MR1229519}). 
When the union of pairs of points of $\Cc$ is dense on $\partial \Ub$, this trace is a rather messy convoluted set, as a given excursion could typically intersect 
infinitely many others. 

The main purpose of the present paper is to explain the proof of the following result: 

\begin{proposition}\label{main theorem}
If one knows the law of the trace of $\Ec$, then one can recover the law of $\Cc$ and the law of $\Ec$. 
\end{proposition}

Let us already say a few words about the strategy of the proof: For all positive $\delta$, let us denote by $\Ec_\delta$ be the finite 
random collection of excursions in $\Ec$ that intersect  the smaller disk $U_\delta := (1- \delta) \Ub$.
For each $E\in\Ec_\delta$, one can define the pair $\{ x', y'\}$ on $\partial U_\delta$  consisting of the first and last points
on the boundary of $U_\delta$ that are visited by an oriented version of $E$. Let us denote by $\Cc_\delta$ the collection of all these pairs. 

Let us now explain why in order to prove Proposition~\ref {main theorem}, it will suffice to show the following lemma: 
\begin{lemma}\label{main lemma}
For all $\delta >0$, the law of $\Cc_\delta$ is determined by the law of the trace of $\Ec$.
\end{lemma}
Indeed, suppose that this lemma holds. If we know the law of the trace of $\Ec$, then we know the law of $\Cc_{1/n}$ for all $n$. There can only be at most countably many $\eps>0$ such that the probability that there exists $\{x,y\}$ in $\Cc$ with $d(x,y)=\eps$ is positive. Hence, by continuity of the excursions, for almost any $\eps>0$, the sets $\{\{ x,y\}\in\Cc_{1/n} \, : \, |x-y| >\eps\}$ converge almost surely (hence in law) 
as $n\to\infty$ to the set $\{\{ x,y\}\in \Cc\, :  \, |x-y| > \eps \}$. This means that the law of $\Cc$ can be recovered.

The strategy of the proof of Lemma~\ref {main lemma} will be (for each fixed $\delta$), to determine the law of 
$\Cc_\delta 1_{ {\mathcal N}_\delta = n}$ inductively over $n$, where ${\mathcal N}_\delta$ denotes the number of excursions in $\Ec_\delta$: 
In Section~\ref {S2}, we will introduce and study some special events that are measurable with respect to the trace $\Ec$, and that loosely speaking 
impose that the part of $\Ec$ that is at distance greater than $\delta$ from the boundary of the disk does stay in a very narrow tube that crosses 
the disk $(1- \delta) \Ub$ (roughly speaking a tube is a small neighborhood of a segment $[x,y]$ where $x$ and $y$ are on the circle of radius $1- \delta$ around 
the origin), and that one excursion does indeed cross the entire tube. We will estimate precisely the asymptotics of their probabilities when the width of the tube vanishes, which 
in turn will allow us to determine the law of $\Cc_\delta 1_{ {\mathcal N}_\delta = 1}$ (in Section~\ref{S3}).
Sections~\ref{S:Two pairs} and~\ref {S4} are devoted to the induction over $n$, which is mostly based on similar ideas, studying the asymptotic probabilities 
of events that $n$ such given tubes are traversed, when the widths of these $n$ tubes vanish. 
We will first focus on non-crossing configurations of tubes i.e., such that no two tubes intersect, and we will then (again inductively) deduce the general case. 
Some simple combinatorial considerations about how several narrow tubes can be traversed by excursions will enable to conclude.

\subsection* {A consequence using earlier work}

The motivation for the present work comes from our paper \cite {Qian-Werner} about Brownian loop-soup cluster decompositions 
and their relation to the conformal loop ensembles. 
This also explains why we focus here on the two-dimensional case (note however that the three-dimensional case is actually easier because 
the issue about the crossing configurations does not arise; the arguments 
of part of this paper can be directly adapted for that case. In higher dimensions, since Brownian excursions are simple paths, the statement is immediate). 

Let us first briefly survey some relevant features from earlier papers (for more references, see \cite {Qian-Werner}) in order to state 
the main consequence that we draw from Proposition~\ref{main theorem}: There exists a natural conformally invariant measure $\mu$ 
on unrooted Brownian loops in the unit disk introduced in \cite {MR2045953}, and for each positive $c$, when one samples a Poisson point process of such loops with intensity exactly $c \mu$, one obtains the so-called {\em Brownian loop-soup} with intensity $c$ (see again \cite {MR2045953}). 

When $c \le 1$, it turns out that
the union of all these loops can be decomposed into infinitely many connected components called loop-soup clusters \cite {MR2023758, MR2979861} and 
the outer boundary of the outermost clusters form what is called a conformal loop ensemble of parameter $\kappa$ in $(8/3, 4]$, where $\kappa$ is some 
explicit function of $c$. 
Let us consider now a Brownian loop-soup with intensity $c \le 1$ in $\Ub$ and choose  
 a given point, say the origin, in $\Ub$. Then, this point will be almost surely surrounded by a CLE$_\kappa$ loop $\partial$ (which is the outer boundary 
of the outermost origin-surrounding cluster $K$ of Brownian loops). Let us denote by $O_\partial$ the simply connected domain encircled by $\partial$. 
The paper  \cite {Qian-Werner} is describing aspects of the conditional distribution of the loop-soup given $\partial$. In particular, the 
Brownian loops inside of $\overline O_\partial$ can be decomposed into two conditionally independent parts: 
(1) The set of Brownian loops in $O_\partial$ that is distributed as a Brownian loop-soup in $O_\partial$. (2) The set of loops that touch $\partial$. 

Note that each loop 
in (2) can be decomposed into a collection of excursions away from $\partial$. One can then map this conformally onto the unit disk via the 
conformal transformation $\phi_\partial$ from $O_\partial$ onto $\Ub$ such that $\phi_\partial (0) = 0$ and $\phi_\partial'(0)$ is a positive real number. 
In this way, one obtains a random collection 
of excursions $\Ec^\partial$ in the unit disk and (see \cite {Qian-Werner}) its law  is invariant under any M\"obius 
transformation of the unit disk. 
As we will explain in Section~\ref {Sfinal}, the results of \cite {Qian-Werner} combined with elementary observations on the Brownian loop-measure do imply that in fact, 
this set of excursions is necessarily of the type $\Ec$ described above. Hence, Proposition~\ref {main theorem} shows that its law is in fact determined by the law of 
its trace. 

The loop-soup with intensity $c=1$ turns out to be very special:  It possesses for 
instance nice resampling properties \cite {MR3618142}, and is very closely related to the GFF in $\Ub$. 
Indeed, the properly renormalized occupation time measure of the union of all these loops turns out to be distributed as the 
(properly defined) square of the GFF \cite {MR2815763}. As shown in \cite {Qian-Werner} (building on the results of \cite {MR3502602,Lupu}), this CLE$_4$ can also be viewed as the collection of outermost 
level-lines (in the sense developed by \cite {MR2486487, MS}) of the Gaussian free field whose square is the occupation time of the loop-soup. 

In that special case, we did show in \cite {Qian-Werner} (building on these relations to the GFF) 
that the law of the trace of $\Ec^\partial$ is identical to the law of the trace of a certain Poisson point process
${\mathcal P}$ of Brownian excursions in $\Ub$.
In particular, the intensity measure on the set of ordered pairs of starting points on the unit circle of this process ${\mathcal P}$ is 
given by 
\begin{align*}
 \frac {d\lambda(x) d\lambda (y)}{4 | x-y|^{2}}, 
\end{align*}
where $\lambda$ is Lebesgue measure on the unit circle (note that, up to a multiplicative constant, this is the only possible measure on 
pairs of points that is invariant under M\"obius transformations -- one important feature of that result is actually the value of the constant; here this is the constant 
such that when restricted to end-points on the half circle, the outer boundary of the Poisson point process of excursions with this intensity 
does create a restriction sample of exponent $1/4$, see \cite {MR2178043}). 
Combining this with Proposition~\ref {main theorem} then implies immediately the following fact: 

\begin{corollary}\label{main corollary}
For the loop-soup with $c=1$, the law of the collection $\Ec^\partial$ is exactly that of the Poisson point process ${\mathcal P}$. 
\end{corollary}

The techniques that 
we used to derive the law of the trace of $\Ec$ in this $c=1$ case were based on Dynkin's isomorphism theorem, 
that provides information on the law of the cumulative occupation times, so that 
the present paper can be used in other contexts where Dynkin's theorem applies. 
For instance, the arguments in the present paper and in \cite {Qian-Werner} can be adapted or go through without further ado  
in order to extend those results to the multiply connected settings (i.e. when one replaces $\Ub$ by a multiply connected domain); see  \cite {ALS} for some aspects of the 
GFF/loop-soup aspects in the non-simply connected setting.  

Note that as explained in \cite {Qian-Werner}, when $c < 1$, one does not expect the law of $\Ec^\partial$ to be that of a Poisson point process (the different excursion should ``interact''). It 
would be nevertheless interesting to understand it better.

\section{Preliminaries} \label{S2}

With the exception of Section~\ref {Sfinal}, 
the remainder of the paper is devoted to the proof of Lemma~\ref {main lemma} and will involve neither Brownian loop-soups nor CLE considerations. 

In this section, we will review some simple facts about oriented excursions and their decomposition  (see \cite{MR2045953} for more details).

Let $D$ be a bounded open domain. We will denote by $G^D (x,y)$ the Green's function in $D$, and we use the normalization so that 
$G^D (x,y) dy $ is the density of the expected occupation time measure for a Brownian motion started from $x$ and killed upon exiting $D$. 
Then, for each $x \not= y$ in $\overline D$, one can define a natural finite measure $ \mu^D(x,y)$ on Brownian paths from $x$ to $y$ in $D$. There are three cases, depending on whether $x$ or $y$ are in $D$ or on the boundary of $D$.

\begin{itemize}
\item For any $x,y\in D$, one can construct the $\mu^D(x,y)$ to be the natural measure on Brownian bridges in $D$ from $x$ to $y$ with total mass $G^D(x,y)$.
This bridge measure is conformally invariant: For any conformal map $f$ from $D$ onto some other domain $D'$, we have 
\begin{align}\label{eq:conformal-inv}
\mu^{D'}(f(x), f(y))=f\circ \mu^D(x,y).
\end{align}
We will refer to this measure renormalized to be a probability measure as the \emph{bridge} probability measure from $x$ to $y$ in $D$. 
\item For $x,y \in \partial D$ such that $\partial D$ is smooth near $x$ and $y$, we define the excursion measure 
$$\mu^D(x,y)=\lim_{\eps\to 0}\frac{1}{\eps^2}\mu^D(x+\eps \vec{n}_x, y+\eps \vec{n}_y)$$
where (here and in the sequel) $\vec{n}_x$ is inwards pointing normal vector to  $\partial D$ at $x$.
Note that this excursion measure is (typically) also not a probability measure, but it has finite mass (its total mass is the boundary Poisson kernel $K^D(x,y)$). 
This time, it is conformally covariant: For any conformal map $f$ from $D$ onto some other domain $D'$ such that $f'(x)$ and $f'(y)$ exist, we have $$|f'(x)| |f'(y)|\mu^{D'}(f(x), f(y))=f\circ \mu^D(x,y).$$
We will refer to this measure renormalized to be a probability measure to be the \emph{excursion} probability measure. This probability measure is then conformally invariant, and 
when $D = \Ub$, it is exactly the excursion probability measure $P^{x,y}$ mentioned in the introduction. 
\item For $x \in D$ and $y \in \partial D$ such that $\partial D$ is smooth in the neighborhood of $z$, we define the measures 
$$\mu^D(x,y)=\lim_{\eps\to 0}\frac{1}{\eps}\mu^D(x, y+\eps \vec{n}_y)$$
and 
$$\mu^D(y,x)=\lim_{\eps\to 0}\frac{1}{\eps}\mu^D(y+\eps \vec{n}_y, x).$$
The latter one is of course obtained by time-reversal of the former. 
Their  total mass $H^D (x,y)$ is now the density at $y$ of the harmonic measure seen from $x$, and  it is conformally covariant: For any conformal map $f$ from $D$ onto some other domain $D'$ such that  $f'(y)$ exists, we have $$ |f'(y)|\mu^{D'}(f(x), f(y))=f\circ \mu^D(x,y).$$
We will refer to these measures (and their renormalized probability measures) as \emph{brexcursions}. 
\end{itemize}
For all these measures, we will omit the superscript $\Ub$ when we will be working in the unit disk $\Ub$ (in other words, we will write $\mu$ instead of $\mu^\Ub$).

It is easy to derive path decompositions of these Brownian paths defined under all these measures (see for instance  \cite{MR2045953}). 
Let us now state here one such  decomposition  that will be relevant for our purpose (and that can be easily proved using the same arguments as in  \cite{MR2045953}).
We will in fact not only be using this particular decomposition, but this one illustrates well how things work. 

We are interested in boundary-to-boundary excursions in the unit disk $\Ub$ that do intersect $U= U_\delta:=(1-\delta)\Ub$ for some given $\delta \in (0,1)$. 
Let $A= A_\delta$ be the open annulus $\Ub\setminus \overline U$. 
Let $\tilde \mu_{x,y}$ be the measure $\mu_{x,y}$ restricted to the set of excursions that intersect $U$ (which is therefore the difference between $\mu_{x,y}$ and 
 $\mu_{x,y}^{A}$).
Then we can decompose those excursions with respect to their first and last visited points $x'$ and $y'$ in $\overline U$ and one can view the excursion as a concatenation of an excursion from $x$ to $x'$ in $A$, a bridge from $x'$ to $y'$ in $\Ub$ and an excursion from $y'$ to $y$ in $A$.  More precisely: 
\begin{lemma}\label{lem:decomp}
The measure $\tilde \mu_{x,y}$ can be decomposed as follows:
\begin{align*}
\tilde \mu_{x,y}=\int_{\partial U_\delta}\int_{\partial U_\delta} (\mu^{A}_{x, x'} \oplus \mu_{x', y'} \oplus \mu^A_{y',y})  \, d \lambda(x') \, d \lambda( y')
\end{align*}
where (here and in the sequel) $d\lambda$  denotes the (one-dimensional) Lebesgue measure (on $\partial U_\delta$).
Here the measure $\mu^{A}_{x, x'} \oplus \mu_{x', y'} \oplus \mu^A_{y',y}$ corresponds to measure on paths obtained by the concatenation of the three different pieces corresponding to
the three measures, when defined under the product measure.  
\end{lemma}

\begin{figure}[h]
\centering
        \includegraphics[trim =60mm 0mm 60mm 0mm, clip,width=0.52\textwidth,page=1]{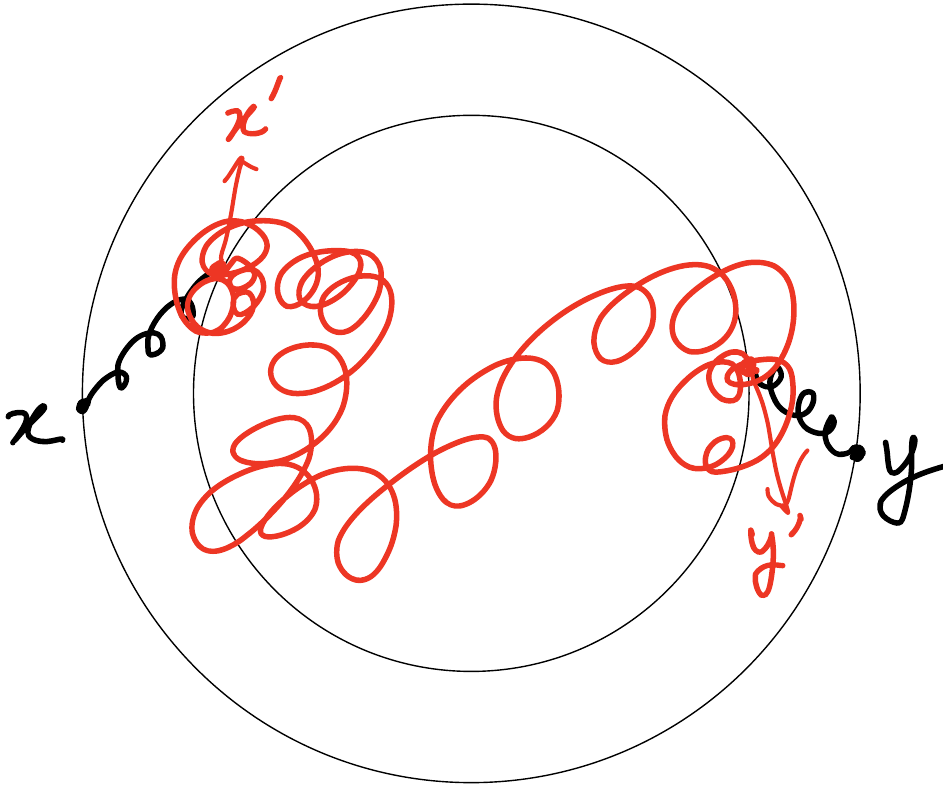}
        \caption{The decomposition of an excursion}
        \label{tube1}
    \end{figure}
    
This decomposition implies in particular the following expression for the total mass of the measure $\tilde \mu_{x,y}$:  
\begin{align*}
\tilde K(x,y) := K(x,y) - K^A (x,y) =\int_{\partial U_\delta}\int_{\partial U_\delta} K^{A} (x, x')  G (x', y') K^A (y',y) \, d \lambda( x')  \, d \lambda( y').
\end{align*}

An alternative way to phrase the lemma is to say that in order to sample a Brownian excursion $E_{x,y}$ in $\Ub$ according to the probability measure on excursions from $x$ to $y$ in $\Ub$ 
conditioned to hit $U_\delta$ (see Figure~\ref{tube1}),  one can first choose $( x', y' )\in \partial U_\delta^2$ according to the density 
$$g(x',y')=\frac {K^A(x,x')G(x',y')K^A(y',y)}{ \tilde K(x,y)}$$ with respect to the product Lebesgue measure, and then conditionally on these two points, draw independently the two excursions and the bridge according to the respective  probability measures.

\medbreak

We will use also some elementary estimates about the boundary Poisson kernel in long tubes. 
Let $\theta\in(0,\pi/2)$. Let $T=T(l)$ be a trapezoid with bottom line $[-l - \pi \cot \theta ,l+\pi\cot\theta]$ and top line $[i \pi  - l , i \pi + l]$. Let $I_1(l)$ and $I_2(l)$ be its left and right sides.

\begin{lemma}\label{trapezoid}
There exists a positive function $c_\theta$ on $(0, \pi )$  such that as $l \to \infty$, 
\begin{align*}
\frac {K^T(x_1,x_2)} {c_\theta (\Im (x_1)) c_\theta(\Im (x_2))}  \underset{l\to\infty}{\sim} \exp(-2l)
\end{align*}
uniformly with respect to  $x_1$ and $x_2$ in $I_1(l)$ and $I_2(l)$ (where here and in the sequel, $x \sim y$ means that $x$ is equivalent to $y$ i.e., that $x/y$ tends to $1$). 
\end{lemma}
\begin{proof}
There are many possible simple ways to derive this fact. Let us give one possibility based on a variant of the previous decomposition of Brownian excursions. 
We consider the vertical segments $L_1= (-l+1 , -l+1+ i \pi)$ and $L_2 = (l-1, l-1 + i \pi)$.
 Let  $R$ denote the rectangle with sides $L_1$ and $L_2$. Then, by considering the last visited point $y_1$ on $L_1$ before the first hitting point $y_2$ of $L_2$, one gets a  
 decomposition of the excursion into a brexcursion from $x_1$ to $y_1$ in $T \setminus L_2$, an excursion from $y_1$ to $y_2$ in $R$ and a brexcursion from $y_2$ to $x_2$ in $T$. 
 The corresponding result for the total mass of the excursion measure is then: 
 \begin{align*}
K^T(x_1,x_2)=\int_{L_1} \int_{L_2} H^{T \setminus L_2}(y_1, x_1) K^R(y_1, y_2) H^{T}(y_2,x_2) d \lambda(y_1) d \lambda(y_2).
\end{align*}
Then, we can just recall that for the Poisson kernel in long rectangles,
\begin{align*}
K^R (y_1, y_2) \underset{l\to\infty}{\sim} c \sin ( \Im (y_1) ) \sin ( \Im (y_2))  \exp(-2l)
\end{align*}
uniformly with respect  to  $ \Im (y_1) $ and $ \Im (y_2)$
(this follows for instance from the expression of the Green's kernel in the upper half-plane and conformal invariance). 
It follows readily that as $l \to \infty$,
\begin{eqnarray*}
K^T(x_1,x_2)&\sim & c  \exp(-2l)  \left( \int_{L_1} \sin ( \Im (y_1) ) H^{T_1}(y_1, x_1)   d \lambda(y_1)  \right) \\
&&  \times \left( \int_{L_2} \sin ( \Im (y_2))   H^{T_2}(y_2,x_2) d \lambda(y_2) \right) .
\end{eqnarray*}
where $T_1$ (resp. $T_2$) denote the union of $T$ with the half strip $ \{\Im(z)\in(0,\pi), \Re (z) > 0 \}$ (resp. with  the half strip $ \{\Im(z)\in(0,\pi), \Re (z) < 0 \}$). 
\end{proof}
We will in fact use the following variation of Lemma~\ref {trapezoid} that can be viewed as a direct consequence of it via the conformal covariance property of the boundary Poisson kernel. 
Suppose that for each $l$, one modifies the trapezoid $T(l)$ into a variant $T'(l)$  by just changing the left 
and right sides $I_1$ and $I_2$ into circular arcs $I_1'$ and $I_2'$ (with the same extremities as $I_1 (l)$ and $I_2 (l) $). We assume that we do this in such a way that 
these circular arcs get closer and closer to $I_1 (l)$ and $I_2 (l)$  as $l \to \infty$. 
Then we can find conformal maps that are very close to the identity which send $T'(l)$ to $T(l')$ for some $l'$ that is very close to $l$.
Then, for the same $c(\theta)$ as in Lemma~\ref {trapezoid}, one has 
\begin{align*}
\frac {K^{T'(l)} (x_1,x_2)} {c_\theta (\Im (x_1)) c_\theta(\Im (x_2))}  \underset{l\to\infty}{\sim} \exp(-2l)
\end{align*}
uniformly with respect to  $x_1$ and $x_2$ in $I_1' (l)$ and $I_2' (l)$.

\medbreak

Finally, let us now mention some features of cut points of Brownian excursions that we will also be using. These are 
exceptional points $z$ on the trace of a Brownian excursion $E$ from $x$ to $y$, such that $E \setminus \{ z \}$ is not connected.
It is known that such points are visited only once by the Brownian excursions (due to the absence of double local cut points 
on Brownian paths \cite {MR1062056} -- we will use this feature again in this paper), and that $x$ and $y$ must then be in different connected components of $E \setminus \{ z \}$. 
It is in fact not difficult to see that a Brownian excursion sampled according to $P_{x,y}$ in the unit disk
has almost surely infinitely many cut points (both near $x$ and near $y$). The set of cut points forms actually a fractal set with dimension $3/4$  \cite {MR1879851} (but we will not need this here). 
By mapping the disk onto a very thin rectangle, one can then directly deduce that if one samples a Brownian excursion onto a very long trapezoid $T(l)$ (or a very long rectangle) and chooses one point on each of its smaller boundary segments, then as $l \to \infty$ (and uniformly with respect to the choice of the boundary points), the probability that an excursion joining these two points has a cut point in any sub portion of $T(l)$ with length $\alpha l$ for any given
$\alpha \in (0,1]$ goes to $1$ as $l \to \infty$. We leave the details to the reader (in fact, it is also not difficult to show that this probability converges exponentially fast to $1$).

\section{One pair}\label{S3}

We now proceed to the proof of Lemma~\ref {main lemma} as outlined in the introduction (we use the same notation). 
From now on, the value of $\delta$ will be fixed. We will write  $U$ and ${\mathcal N}$ instead of  $U_\delta$ and ${\mathcal N}_\delta$.

Note that for any integer $n$, $\Pb({\mathcal N}=n ) > 0$ if and only if the probability 
that $\Cc$ has at least $n$ pairs of points is positive. In particular, if $\sum_{ j < n} \Pb({\mathcal N}=j ) < 1$, then  $\Pb({\mathcal N}=n ) > 0$.

Observe first that $\Pb({\mathcal N}=0)$ is the probability that the trace of $\Ec$ does not intersect $U$, and it is therefore 
 determined by the law of the trace of $\Ec$. We are from now on going to assume that this probability is not equal to $1$ (i.e., that 
 $\Ec$ is not almost surely empty), 
and the goal of the rest of this section will be to show that 
the law of $\Cc_\delta 1_{{\mathcal N} = 1 }$ can be deduced from the law of the trace of $\Ec$. 
In the next sections, 
we will then describe the steps that allows to determine the law of $\Cc_\delta 1_{{\mathcal N} = n }$ for all $n > 1$.

On the event ${\mathcal N} =1$, we denote by  $\{ x', y'\}$ the unique pair in $\Cc_\delta$. It will
also be convenient to consider the ordered pair of points obtained by assigning either order with probability $1/2$.
The conditional law of this ordered pair given ${\mathcal N} =1$ has a smooth positive and symmetric density $f(a,b)$ with respect to the Lebesgue measure 
on $\partial U^2$ (simply because it 
is the mean value of uniformly smooth densities, when one averages with respect to the law of the end-points of the excursion that actually makes it to $U$). 
So, our goal here is to recover $\Pb ( {\mathcal N} = 1 )$ as well as the density $f$ from the law of the trace of $\Ec$. 
 
For two distinct points $a,b\in \partial U$ and for $\eps>0$, let $T^\eps(a,b)$ be the open tube which is equal to  $L^\eps(a,b)\cap U$ where $L^\eps(a,b)$ is the $\eps/2$-neighborhood of the line passing through the points $a,b$ (see Figure~\ref{tube2}). 
Let $[a]$ and $[b]$ denote respectively the two arcs (for a given $a$ and $b$, they exist provided $\eps$ is small enough) of $L^\eps(a,b)\cap\partial U$
that respectively contain $a$ and $b$. 

We define $A^\eps_{a,b}$  to be the event that the following three events (i)-(ii)-(iii) do hold: 

(i) $\Tc(\Ec) \cap U \subset T^\eps(a,b)$,

(ii) $ ( \Tc(\Ec) \cap U) \cup [a] \cup [b]$ is connected,

(iii) $ ( \Tc(\Ec) \cap U) \cup [a] \cup [b]$ has a cut point disconnecting $[a]$ from $[b]$ (which means that there exists some point $z$ when one removes this point $z$, $[a]$ and $[b]$ are in different connected components 
of $ ( \Tc(\Ec) \cap U) \cup [a] \cup [b]$). 

Note that the event $A_{a,b}^\eps$ is by definition measurable with respect to the trace of $\Ec$. 
We introduced this cut point condition (iii) because it implies that 
the tube is crossed exactly once by only one excursion. 
Let us explain this: Recall (see \cite{MR1062056}) that almost surely, there are no double cut points on a planar Brownian path $B[0,T]$ when $T$ is a given deterministic time, hence almost surely, 
any cut point  $x$ of the set $B[0,T]$  (which means that $B[0,T] \setminus \{ x\}$ has two connected components) is visited only once by $B$ on $[0,T]$.  
Simple absolute continuity considerations then imply that if $B$ and $B'$ are two independent Brownian motions (started from any two given points in the plane) on some given time-intervals $[0,T]$ and $[0,T']$, 
then almost surely on the event that $B[0,T] \cup B'[0,T']$ 
is connected, any cut point of this union does belong to only one of the two sets $B[0,T]$ or $B'[0, T']$, and is visited only once by the corresponding Brownian motion. 
We can again invoke some absolute continuity arguments between portions of Brownian excursions and Brownian motions (and the fact that there are only countable many excursions in the Poisson point process of excursions) 
to readily deduce that  (iii) excludes the possibility that the 
 tube is crossed more than once by the same Brownian excursion (as the two crossings would have to visit this cut point). Similarly, (iii) also excludes the possibility 
 that the tube is crossed by more than one excursion (otherwise, both crossings would have to go through this same cut point). Finally,  (iii)  excludes the possibility 
 that there is no crossing at all of the tube by one excursion 
(this would mean that one portion of an excursion entering and leaving the tube from $[a]$ and another portion 
of an excursion entering and leaving from $[b]$ do intersect somewhere in the tube to form a connected set that contain a cut point; but then, this cut point would either 
have to be visited more than once by one of the two 
portions, or it would have to be visited by both portions, and both these possibilities are excluded by the previous observations).

However, $A^\eps_{a,b}$ does not exclude the possibility that  ${\mathcal N}>1$ (it just implies that only one of these ${\mathcal N}$ excursions contains a crossing of $T^\eps(a,b)$).

\begin{figure}[h]
\centering
\includegraphics[trim = 60mm 0mm 60mm 0mm, clip, width=0.45\textwidth, page=2]{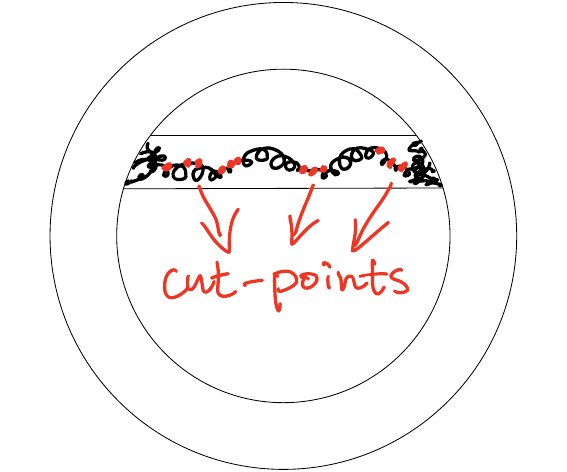}
\caption{The event $A_{a,b}^\eps$}
\label{tube2}
\end{figure}

The law of $\Cc_\delta 1_{{\mathcal N} = 1 }$ will be determined from the following lemma: 
\begin{lemma}\label{key-lemma}
For all distinct $a,b\in \partial U$, there is a constant $c_0(a,b)$ depending solely on $a,b$ (i.e. it does not depend on the law of $\Ec$) such that
$$\Pb(A^\eps_{a,b})\underset{\eps\to 0}{\sim} \Pb({\mathcal N}=1) \times f(a,b) \times c_0(a,b) \times \eps^2 \times  \exp(-\pi |a-b|/\eps) .$$
\end{lemma}
Indeed, this lemma enables (if we know the law of the trace of $\Ec$) to determine the function $(a,b) \mapsto \Pb({\mathcal N}=1)f(a,b)$. 
Let us emphasize that $c_0(a,b)$ can be explicitly expressed in terms of a few poisson kernels (but the existence of $c_0 (a, b)$ is all what we need for our purposes).
We can then deduce the value of $\Pb({\mathcal N}=1)$ by integrating it over $\partial U^2$ and finally determine the function $f$. 

Let us now prove Lemma~\ref {key-lemma}: 

\begin{proof}

It suffices to consider the case where  $a=(1-\delta)e^{i(\pi/2+\theta)}$ for some $\theta\in(0,\pi/2]$ and $b = - \overline a$ (the general case is then obtained by a rotation around the origin).
In the rest of this proof, $\theta$, $a$ and $b$ will be fixed. 
We call $AT^\eps$ the domain obtained by glueing the tube $T^\eps$ to the annulus $A$ (more precisely, it is 
the interior of the closure of $A \cup T^\eps$).

 We will again consider the oriented versions of the  excursions $\Ec$ in this proof (obtained by assigning random orientations to the excursion in $\Ec$ with i.i.d. fair coins).

 The main part of the proof is to show that there exists $c_0(a,b)$ (independent of the law of $\Ec$) such that
\begin{align}\label{eq1}
\Pb(A^\eps_{a,b} \mid  {\mathcal N}=1)\underset{\eps\to 0}{\sim} f(a,b) \times  c_0(a,b) \times \eps^2 \times \exp\left({ -\pi |a-b|}{\eps} \right).
\end{align}

Conditionally on $\{{\mathcal N}=1\}$, let us call 
$x'$ and $y'$ the first and last points on $\partial U$ visited by that unique excursion, and let us call $E'$ the middle bridge from $x'$ to $y'$ by that unique excursion. 
Note that by the decomposition of excursions recalled in Section~\ref {S2}, the conditional law of $E'$ given ${\mathcal N}=1$, $x'$ and $y'$ is the bridge from $x'$ to 
$y'$ in $\Ub$. 

Furthermore, the conditional distribution of $(x', y')$ is then described 
by $f$. Since $f$ is smooth, the probability that $x' \in [a]$ and $y' \in [b]$ will be equivalent to $\lambda ([a]) \lambda ([b]) f(a,b) $ as $\eps \to 0$, and 
conditionally on this event, the conditional distribution of $(x', y')$ will be close to uniform on $[a] \times [b]$. 

Note that conditionally on ${\mathcal N}=1$, the event $A^\eps_{a, b} \cap \{{\mathcal N} = 1 \}$ can be read off from $E'$. 
Furthermore, (because of the considerations on cut points), when 
$ A^\eps_{a,b} \cap \{   {\mathcal N}=1 \} $ occurs, it means that the bridge $E'$ does cross the tube $T^\eps (a,b)$ exactly once. Note that this crossing can either be from the left to the right, or from the right to the left (these two events are not symmetric, since we chose our bridge $E'$ to start from $x'$ and to finish at $y'$). 

Conditionally on ${\mathcal N}=1$,  let us denote by $\tilde A^\eps  
= \tilde A^\eps_{a,b}$ the event that $E'$ stays in $AT^\eps$ and does cross the tube $T^\eps (a,b)$ exactly once and that it does so in the direction from 
left to right. 
We are now going to evaluate the probability of $\tilde A^\eps_{a.b} \cap \{ {\mathcal N} =1 \}  $ and  argue that 
\begin{align} \label{aab}
 \Pb  [ A^\eps_{a,b} \cap \{    {\mathcal N} =1   \} ]  \sim  \Pb [  \tilde A^\eps_{a.b} \cap \{  {\mathcal N } = 1 \}  ]
\end{align}
as $\eps \to 0$. 

Let us work conditionally on ${\mathcal N}=1$, $x'$ and $y'$. 
On the event $ \tilde A^\eps_{a,b}$, we can decompose $E'$ into three (conditionally) independent parts 
(see Figure~\ref{tube3}):
\begin {itemize}
 \item A bridge from $x'$ to some $x'' \in [a]$ with the property that it does not cross  $T^\eps$. 
 \item An excursion in $T^\eps$ from $x''$ to some $y'' \in [b]$.
 \item A bridge from $y''$ to $y'$ in $AT^\eps$ with the property that it does not cross $T^\eps$. 
\end {itemize}
The idea is now the following: 
A bridge from $x'$ to $x''$ in $AT^\eps$ will typically stay pretty close to $a$. 
In particular, with a (conditional) probability that tends to $1$ when $\eps \to 0$, it will not reach the middle third of $T^\eps$. 
The similar result will hold for the final bridge between $y''$ and $y'$. 
On the other hand, an excursion in $T^\eps$ from $x''$ to $y''$ will typically contain a cut point in the middle forth of $T^\eps$. 
Hence, when $\eps \to 0$, the conditional probability given $ \tilde A^\eps_{a,b} \cap \{ {\mathcal N}=1 \}$ that $A^\eps (a,b)$ holds 
goes to $1$. 

Furthermore, the previous decomposition allows to evaluate the probability that $\tilde A^\eps_{a,b}$ holds, conditionally on ${\mathcal N} =1$, $x'$ and $y'$:  
One has to integrate the product of the following three terms with respect to the Lebesgue measure over $x''$ and $y''$ on $[a]$ and $[b]$: 
\begin {enumerate} 
 \item The mass of the bridges from $x'$ to $x''$ in $AT^\eps$ that do not cross $T^\eps$. 
 \item The mass of the bridges from $y''$ to $y'$ in $AT^\eps$ that do not cross $T^\eps$. 
 \item The mass of the excursion measure from $x''$ to $y''$ in $T^\eps$. 
\end {enumerate}
We can already note that the total mass of the excursion measure from $x''$ to $y''$ is (modulo scaling) controled by the consequence 
of Lemma~\ref {trapezoid}. The limiting behavior of the masses of the bridge measures is also easily described using scaling (by $1/\eps$). 
Let $W$ be the domain depicted in the left of Figure~\ref{tube4}, which is the union of the half-plane to the left of the line $e^{i\theta}\Rb$ and a horizontal strip of width $\eps$.
As $\eps\to 0$, $G^{AT^\eps}(x',x'')$ converges to $G^W(\tilde x',\tilde x'')$ where $\tilde x', \tilde x''$ are the images of $x', x''$ under scaling. Here is a simple proof: For any $r>0$, the domain $B(a,\eps r)\cap AT^\eps$ rescaled by $1/\eps$ converges to $B(a,r)\cap W$, hence $G^{B(a,\eps r)\cap AT^\eps}(x',x'')$ converges to $G^{B(a,r)\cap W}(\tilde x',\tilde x'')$ uniformly for all $x',x''\in[a]$ (by conformal invariance of $G$). On the other hand,  the difference between $G^{B(a,\eps r)\cap AT^\eps}(x',x'')$ and $G^{AT^\eps}(x',x'')$ as well as  that between $G^{B(a,r)\cap W}(\tilde x',\tilde x'')$ and $G^W(\tilde x',\tilde x'')$ is bounded by some function $c(r)$ which goes to $0$ as $r\to\infty$, for all $r<\delta/\eps$ and for all $x',x''\in[a]$.
\begin{figure}[h!]
\centering
\includegraphics[trim = 60mm 36mm 60mm 0mm, clip, width=0.56\textwidth, page=3]{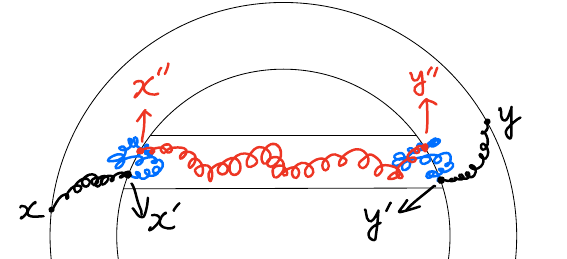}
\caption{Decomposition of the excursion on the event $\tilde A_{x,y}^\eps$}
\label{tube3}
\end{figure}
\begin{figure}[h!]
\centering
\includegraphics[trim = 51mm 0mm 46mm 0mm, clip, width=\textwidth]{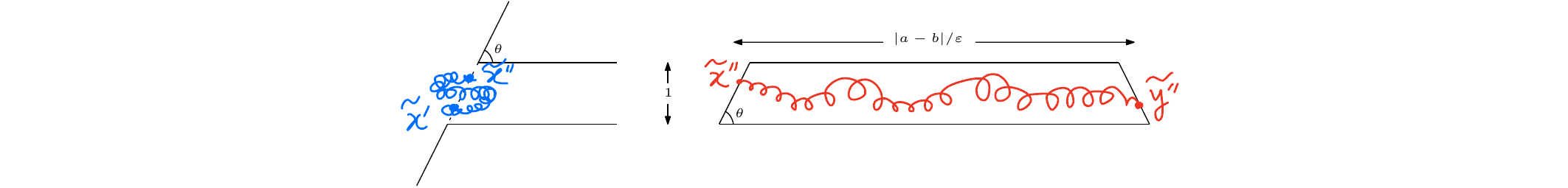}
\caption{The regions $W$ and $V^\eps$}
\label{tube4}
\end{figure}
{
We can then integrate the product of the three measures and get that the mass of the bridge measure from $x'$ to $y'$ restricted to the event $\tilde A_{a,b}$ is equal to
$c(\tilde x',\tilde y') \exp(-\pi|a-b|/\eps)$ where $c(\tilde x', \tilde y')$ is some constant (independent of the law of $\Ec$). Integrating again this times $f(x', y')$ with respect to $d\lambda(x') d\lambda(y')$ on $[a]\times [b]$, we get that there exist some constant $c_0(a,b)$ such that
\begin{align*}
\Pb(\tilde A^\eps_{a.b} \mid  {\mathcal N}=1)\underset{\eps\to 0}{\sim} f(a,b) \times  c_0(a,b) \times \eps^2 \times \exp\left({ -\pi |a-b|}/{\eps} \right).
\end{align*}
}
We then want to deduce (\ref{aab}). 
For this, it suffices to show that conditionally on $\Nc=1$, the probability of crossing  $T^\eps$ in the wrong direction is of smaller order than  $\eps^2 \exp(-\pi |a-b|/\eps)$. If the excursion crosses $T^\eps$ in the wrong direction, then it first needs to intersect $[a]$ without crossing $T^\eps$ (which costs $O(\eps^2)$), then it needs to intersect $[b]$ and make the crossing from $[b]$ to $[a]$ (which costs a constant times $\eps^2\exp(-\pi |a-b|/\eps)$) and finally it needs to intersect $[b]$ without crossing $T^\eps$ (which again costs $O(\eps^2)$). The (conditional) probability of this event is of order $\eps^6 \exp(-\pi |a-b|/\eps)$ as $\eps \to 0$. 
We have thus proved (\ref{aab}) and (\ref{eq1}).

It finally remains to explain why $\Pb(A^\eps_{a,b} , {\mathcal N} \ge 2 )$ decays much faster than   $\Pb(A^\eps_{a,b}, {\mathcal N} =1  )$.
For $k>1$, if $\Pb( {\mathcal N}=k)>0$, then on the event $A^\eps_{a,b} \cap \{ {\mathcal N} =k\}$, one of the $k$ pairs in $\Ec_\delta$ has to cross the tube $T^\eps$ and the other $k-1$ pairs have to intersect $T^\eps$ without crossing it, because of the cut point condition (iii). For one excursion, crossing $T^\eps$ costs a constant times $ \eps^{2} \exp(-\pi |a-b|/\eps)$ (for same reason as (\ref{eq1})) and intersecting $T^\eps$ without crossing costs a constant times $\eps^2$. Therefore,
$$\Pb(A^\eps_{a,b}, {\mathcal N}=k) =O\left[   \eps^{2k} \exp(-\pi |a-b|/\eps) \right].$$
Since $\Pb(A^\eps_{a,b})=\sum_{k=0}^\infty \Pb(A^\eps_{a,b}, {\mathcal N}=k)$, this concludes the proof.
\end{proof}

\section{Two pairs} \label{S:Two pairs}

If $\Pb(\Nc=1)<1$, then we know that $\Pb(\Nc=2)>0$.
Conditionally on $\Nc=2$, we first choose either order between the two pairs at random (with probability $1/2$) and then assign either order between the two points of each pair with probability $1/2$ as well. In this way, we obtain an ordered pair of ordered pairs $((x_1', y_1'), (x_2', y_2'))$ in $(\partial U)^4$. 
Just as in the previous section, the conditional law of the quadruplet given $\Nc=2$ then admits a smooth and positive density function $f_2$.
The goal of this section is to deduce $\Pb(\Nc=2)$ and $f_2$ from the law of the trace of $\Ec$. 
We will detail only the aspects of the proof that are new compared to the case of one pair.

\subsection*{Two non-crossing pairs}

We first consider two pairs $(a_1,a_2), (a_3, a_4)$ of distinct points on $\partial U$ such that the segments $[a_1, a_2]$ and $[a_3, a_4]$ do not cross each other. Then for $\eps$ small enough, the two tubes $T^\eps(a_1,a_2)$ and $T^\eps(a_3, a_4)$ do not intersect each other (see Figure~\ref{tubeb3}).
We define $A^\eps(a_1,a_2,a_3,a_4)$ to be the event that (i) $\Tc(\Ec)\cap U\subset T^\eps(a_1,a_2)\cup T^\eps(a_3, a_4)$, that (ii) $(\Tc(\Ec)\cap T^\eps(a_1,a_2))\cup[a_1]\cup[a_2]$ is connected, but has a cut point disconnecting $[a_1]$ from $[a_2]$ and that (iii) $(\Tc(\Ec)\cap T^\eps(a_3,a_4))\cup[a_3]\cup[a_4]$ is connected, but has a cut point disconnecting $[a_3]$ from $[a_4]$.
Note that this event $A^\eps(a_1,a_2,a_3,a_4)$ is measurable with respect to the trace of $\Ec$.

We first aim to show the following lemma.
\begin{lemma}\label{lem1234}
We have
\begin{align*}
\Pb(A^\eps(a_1,a_2,a_3,a_4))\underset{\eps\to 0}{\sim}\Pb(A^\eps(a_1,a_2,a_3,a_4), \Nc\le 2).
\end{align*}
\end{lemma}

The proof of Lemma~\ref{lem1234} will use the following lemma which will build on the results of Section~\ref{S3}.
\begin{lemma}\label{lem8}
There exist some function $\tilde f$ which is determined by the law of the trace of $\Ec$, such that 
\begin{eqnarray*}
\lefteqn {\Pb(A^\eps(a_1,a_2,a_3,a_4),  \Nc=1)} 
\\
&\underset{\eps\to 0}{\sim}&  \tilde f(a_1,a_2,a_3,a_4) \times \eps^4 \times   \exp(-\pi |a_1-a_2|/\eps)  \times  \exp(-\pi |a_3-a_4|/\eps).
\end{eqnarray*}
\end{lemma}
Let us first show how to deduce Lemma~\ref{lem1234} from Lemma~\ref{lem8} using the same ideas as in Section~\ref {S3}.
\begin{proof}[Proof of Lemma~\ref{lem1234}]
Conditionally on $\Nc=2$ and on $\Dc$, the two Brownian bridges in $\Ub$ that respectively connect the two pairs of points in $\Dc$ are independent.
We can directly adapt Lemma~\ref{key-lemma} and (\ref{eq1}) to get that
\begin{equation}\label{eq1234}
\begin{split}
\Pb(A^\eps(a_1,a_2,a_3,a_4) \mid \Nc=2)\underset{\eps\to 0}{\sim}  & 8f_2(a_1,a_2,a_3,a_4) \times c_0(a_1, a_2)  \times \eps^2 \times  \exp(-\pi |a_1-a_2|/\eps)\\
&\times c_0(a_3, a_4)  \times \eps^2 \times  \exp(-\pi |a_3-a_4|/\eps),
\end{split}
\end{equation}
where the factor $8$ is due to the permutations of the points $a_1,a_2,a_3,a_4$. 

It is also immediate to see that $\Pb(A^\eps(a_1,a_2,a_3,a_4), \Nc\ge 3)$ is of smaller order, because at least one of the excursions needs to intersect one of the tubes without crossing it, creating an extra $\eps^2$ factor in the probability. 
\end{proof}
One can then conclude just as in Section~\ref {S3}: Combining Lemma~\ref{lem1234}, Lemma~\ref{lem8} and (\ref{eq1234}) enables to determine  
$$\Pb(\Nc=2) f_2(a_1,a_2,a_3,a_4)  $$  as a function of the law of the trace of $\Ec$ (but only for all non-crossing pairs $(a_1,a_2), (a_3,a_4)$).
It therefore remains to explain the proof of Lemma~\ref {lem8}. 

\begin{proof}[Proof of Lemma~\ref{lem8}]

\begin{figure}[h!]
\centering
  \includegraphics[trim = 65mm 0mm 65mm 0mm, clip,width=.5\linewidth]{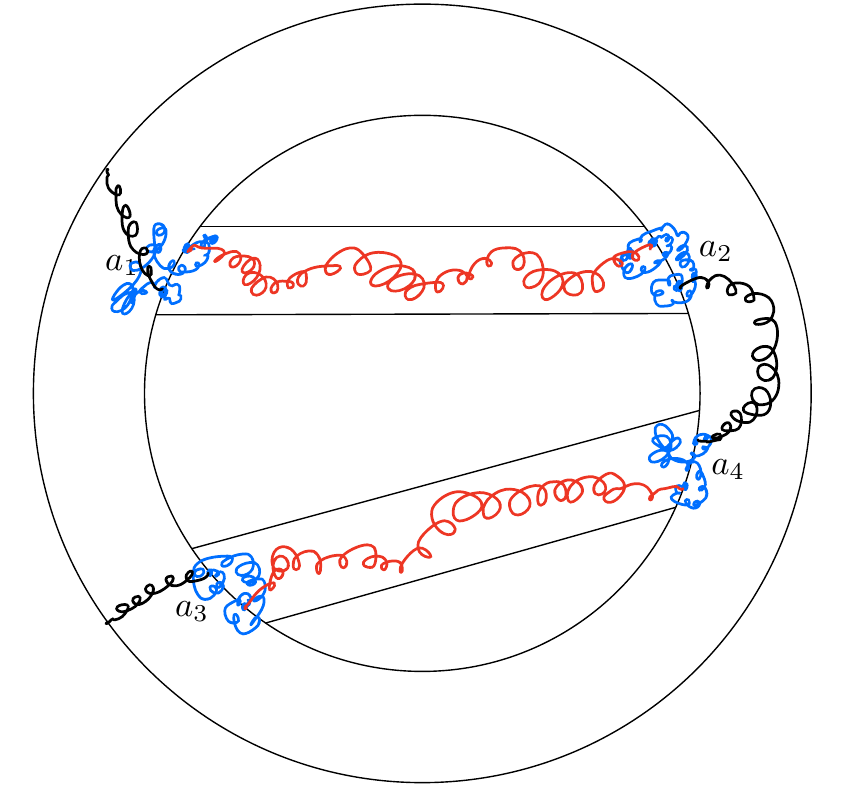}
  \captionof{figure}{Decomposition of the excursion for the configuration $\tilde C(a_1,a_2,a_4,a_3)$}
  \label{tubeb3}
\end{figure}

We will consider the oriented versions of the excursions in $\Ec$ (by assigning either orientation to the excursions in $\Ec$ with probability $1/2$).
The event $A^\eps(a_1,a_2,a_3,a_4) \cap \{ \Nc=1\}$ is then the union of eight configurations corresponding to choice of the order and of the directions in which 
the two tubes are crossed by the oriented excursion. 
We define $C(a_i,a_j,a_k,a_l)$ to be the event such that the (unique) excursion in $\Dc$ first crosses $T^\eps(a_i,a_j)$ from $[a_i]$ to $[a_j]$ and then crosses $T^\eps(a_k,a_l)$ from $[a_k]$ to $[a_l]$.
 Since a configuration and its time reversal have the same probability, we get that
\begin{align*}
\Pb\left[A^\eps(a_1,a_2,a_3,a_4), \Nc=1 \right]= &2\Pb[C(a_1,a_2,a_4,a_3)]+ 2\Pb[C(a_1,a_2,a_3,a_4)]\\
&+2 \Pb[C(a_2,a_1,a_3,a_4)]+2\Pb[C(a_2,a_1,a_4,a_3)].
\end{align*}
It will therefore be sufficient to estimate the quantities $\Pb[C(a_i,a_j,a_k,a_l)]$.

Instead of $C(a_1,a_2,a_4,a_3)$, we consider a similar event $\tilde C(a_1,a_2,a_4,a_3)$ 
where the excursion stays in $A\cup T^\eps(a_1,a_2) \cup T^\eps(a_3,a_4)$, first hits $U$ on some point $x'\in [a_1]$, then crosses the tube $T^\eps(a_1,a_2)$ from $[a_1]$ to $[a_2]$, then goes from $[a_2]$ to $[a_4]$ without returning to $[a_2]$ after hitting $[a_4]$, then it crosses $T^\eps(a_3,a_4)$ from $[a_4]$ to $[a_3]$ and finally it ends at $\partial\Ub$ without intersecting again $[a_1], [a_2], [a_4]$.
It is easy to see, using the same ideas than in Section~\ref{S3}, that
\begin{align*}
\Pb(C(a_1,a_2,a_4,a_3))\underset{\eps\to 0}{\sim} \Pb(\tilde C(a_1,a_2,a_4,a_3)).
\end{align*}
We can then repeat similar arguments as in Lemma~\ref{key-lemma}:

The conditional distribution of $(x',y')\in \Dc$ has the density function $f$ as defined in Section~\ref{S3}. Therefore, the conditional probability that $x'\in[a_1]$ and $y'\in[a_3]$ is equivalent to $\lambda[a_1] \lambda[a_3] f(a_1,a_3)$ as $\eps\to 0$.  

On the event $\tilde C(a_1,a_2,a_4,a_3)$ , we can decompose the bridge from $x'$ to $y'$ in the following way (see Figure~\ref{tubeb3}): 
\begin{enumerate}
\item A bridge from $x'$ to $x''$ in $A\cup T^\eps(a_1,a_2)$ that does not cross $T^\eps(a_1,a_2)$
\item An excursion from $x''$ to $w''$ in $T^\eps(a_1,a_2)$
\item A bridge from $w''$ to $w'$  in $A\cup T^\eps(a_1,a_2)$ that does not cross $T^\eps(a_1,a_2)$
\item An excursion from $w'$ to $z'$ in $A$
\item  A bridge from $z'$ to $z''$  in $A\cup T^\eps(a_3,a_4)$ that does not cross $T^\eps(a_3,a_4)$
\item An excursion from $z''$ to $y''$ in $T^\eps(a_3,a_4)$
\item A bridge from $y''$ to $y'$  in $A\cup T^\eps(a_3,a_4)$ that does not cross $T^\eps(a_3,a_4)$
\end{enumerate}
We then proceed as in Lemma~\ref{key-lemma}. After rescaling and integration, one gets that there exist some constant $c(a_1,a_2,a_4,a_3)$
{(here and in the sequel, except stated otherwise, the constants are all independent of the law of $\Tc(\Ec)$)} such that for the same density function $f$, 
\begin{align*}
\Pb(C(a_1,a_2,a_4,a_3))\underset{\eps\to 0}{\sim} &\Pb(\Nc=1) \times f(a_1,a_3)\times c(a_1,a_2,a_4,a_3)\\
& \times \eps^4\times \exp(-\pi|a_1-a_2|/\eps) \times \exp(-\pi |a_3-a_4|/\eps).
\end{align*}
Similarly, the probabilities of the other configurations are also equivalent to a constant times $f$ (applied to the corresponding pair) times $\eps^4\times \exp(-\pi|a_1-a_2|/\eps) \times \exp(-\pi |a_3-a_4|/\eps)$. Roughly speaking, the term $\eps^4\times \exp(-\pi|a_1-a_2|/\eps) \times \exp(-\pi |a_3-a_4|/\eps)$ 
comes from the contributions of the two bridges crossing the two tubes, and the different constant factors in front { (including $f$)} come from the masses of
the excursions in $A$. Summing up all the terms, the lemma follows.
\end{proof}

\subsection*{Two crossing pairs}

We now consider the case when the two segments  $[a_1, a_2]$ and $[a_3, a_4]$ are not disjoint (but we still suppose that these four points are all different).
As we will explain in Section~\ref {Sfinal}, in the very special setting of Corollary~\ref {main corollary}, one can take advantage of the special resampling properties 
of the Brownian loop-soup with parameter $c=1$ in order to by-pass the study of this case, and to obtain it as a direct consequence of the case of two non-crossing pairs. 
In other words, the following part is needed only to derive Proposition~\ref {main theorem} in the general case. 

\subsubsection*{Setup and notation}
When $\eps$ is small enough, the union of the tubes $T^\eps(a_1,a_2)$ and $T^\eps(a_3,a_4)$ can subdivide each other into four sub-tubes $I_1,I_2,I_3,I_4$ and the intersection of the two tubes (which is a small 
rhombus) as depicted in Figure~\ref{sub-tube2}. 
Let $[a_1'], [a_2'], [a_3'], [a_4']$ be the four sides of the middle rhombus, in such a way that 
each sub-tube $I_i$ has two ends $[a_i]$ and $[a_i']$ for $i=1,2,3,4$.

We define $A^\eps(a_1,a_2,a_3,a_4)$ to be the event such that 
(i) $\Tc(\Ec)\cap U\subset T^\eps(a_1,a_2)\cup T^\eps (a_3,a_4)$ and that (ii) $(\Tc(\Ec)\cap I_i) \cup [a_i] \cup [a_i']$ is connected and has a cut point disconnecting $[a_i]$ from $[a_i']$ for all $i=1,2,3,4$.
The event $A^\eps(a_1,a_2,a_3,a_4)$ is measurable with respect to the trace of $\Ec$.

\begin{figure}[h]
\centering
  \includegraphics[width=.40\linewidth, page=10]{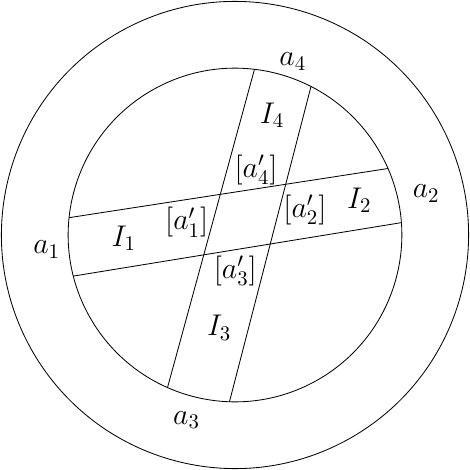}
  \captionof{figure}{The tubes and sub-tubes}
  \label{sub-tube2}
  \end {figure}
  
\begin {figure}
  \centering
  \includegraphics[width=.75\linewidth, page=4]{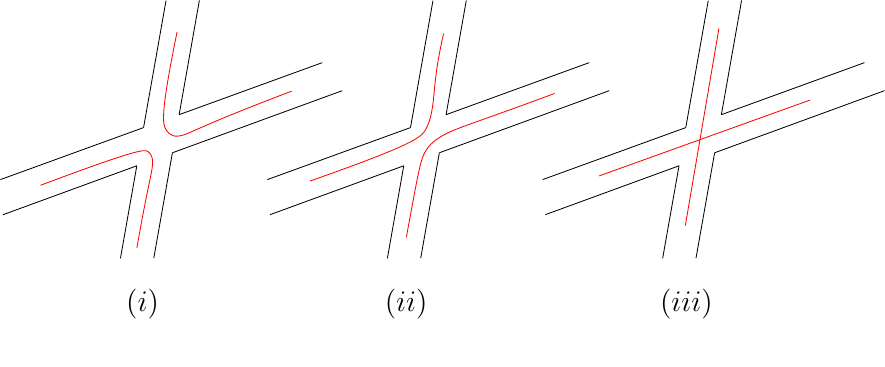}
  \captionof{figure}{The three possible patterns at a crossroad}
  \label{fig:cross2}
\end{figure}

On the event $A^\eps(a_1,a_2,a_3,a_4)$, the union of the excursions in $\Fc$ crosses each sub-tube $I_i$ exactly once. At the crossroad, the excursions in the sub-tubes can connect into each other in three possible ways (see Figure~\ref{fig:cross2}).
 Each way gives rise to two bridges and two pairs of endpoints among  $x_1'\in[a_1],x_2'\in[a_2], x_3'\in[a_3], x_4'\in[a_4]$. 
Then the two bridges are connected by excursions in the outer annulus $A$ into the excursion(s) in $\Fc$. 
Conditionally on $x_1', x_2', x_3', x_4'$ and on the pairing, the bridges and the excursions in $A$ are independent.
We can therefore  separately compute the masses of the bridge measures and the masses of the excursion measures while distinguishing different patterns for both cases. Then, for each combination of patterns, we integrate the product of the masses with respect to $d\lambda(x_1')d\lambda(x_2')d\lambda(x_3')d\lambda(x_4')$ on $[a_1]\times [a_2] \times [a_3]\times [a_4]$ and then sum them up.

\subsubsection*{The masses of the bridges}
We will illustrate the computation on pattern (i) (the same idea works for the other two patterns). 
For pattern (i), we need to compute the masses of two bridge measures that stay in $T^\eps(a_1,a_2) \cup T^\eps (a_3, a_4)$, one from $x_1'$ to $x_3'$, the other from $x_4'$ to $x_2'$. We will only compute here the mass of the bridge measure from $x_1'$ to $x_3'$ (the computation for the bridge measures from $x_4'$ to $x_2'$ works again in the same way). In the $\eps\to 0$ limit, this bridge measure can be decomposed into the following measures (see Figure~\ref{tubeb4})

\begin{figure}[h!]
\centering
\includegraphics[trim = 60mm 0mm 60mm 0mm, clip, width=0.55\textwidth]{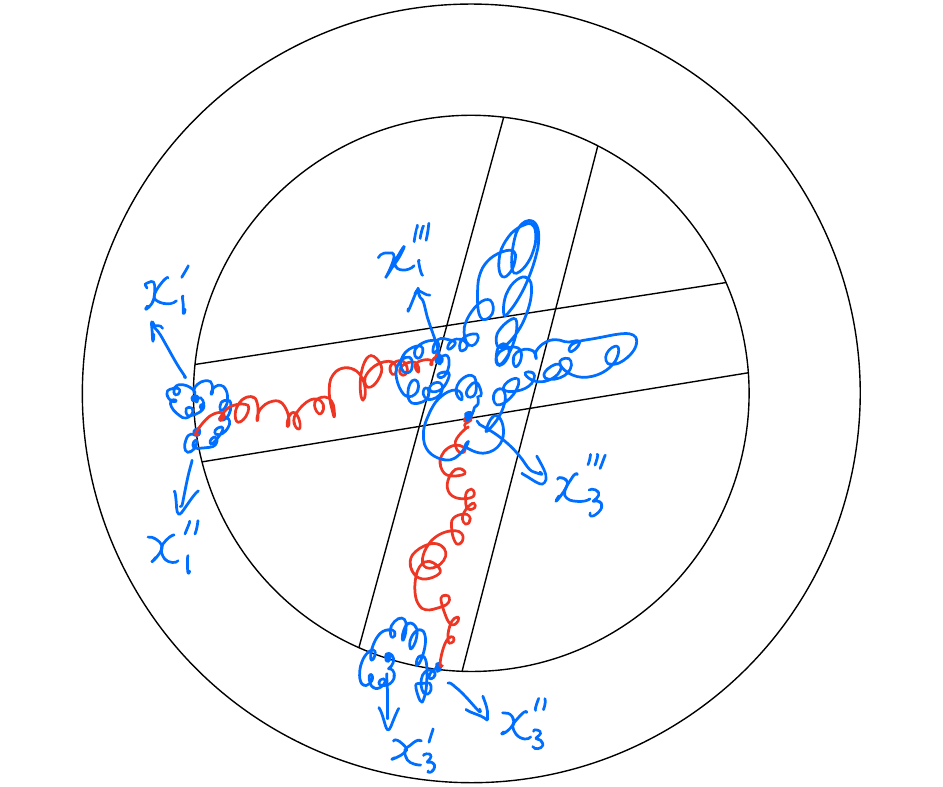}
\caption{Decomposition of the bridge measure}
\label{tubeb4}
\end{figure}

\begin{enumerate}
\item A bridge measure from $x_1'$ to $x_1''$ in $A\cup T^\eps(a_1,a_2)$ without crossing $T^\eps(a_1,a_2)$
\item An excursion measure from $x_1''$ to $x_1'''$ in $I_1$
\item A bridge measure from $x_1'''$ to $x_3'''$ in $T^\eps(a_3,a_4) \cup T^\eps(a_1,a_2)$ without crossing $T^\eps(a_1,a_2)$ or $T^\eps(a_3, a_4)$
\item An excursion measure from $x_3'''$ to $x_3''$ in $I_3$
\item A bridge measure from $x_3''$ to $x_3'$  in $A\cup T^\eps(a_3,a_4)$ without crossing $T^\eps(a_3,a_4)$
\end{enumerate}
After rescaling and integration, we get that the mass of the bridge from $x_1'$ to $x_3'$ is equivalent to a constant factor (depending on the $x_1'$ and $x_3'$) times 
$
\exp(-\pi (|I_1|+|I_3|)/\eps).
$
Similarly, the mass of the bridge from $x_1'$ to $x_3'$ is equivalent to a constant factor (depending on the $x_2'$ and $x_4'$) times 
$
 \exp(-\pi (|I_2|+|I_4|)/\eps),
$
where $|I_i|$ denotes the length of the tube $I_i$.

In fact, for all of the three patterns, the product of the masses of the two corresponding bridges will be equivalent to a constant factor (depending on $X':=(x_1', \ldots, x_4')$ and on the pattern) times $\exp(-\pi (|I_1|+|I_2|+|I_3|+|I_4|)/\eps)$.

\subsubsection*{The masses of the excursions in $A$} 
Given a crossroad pattern (thus also a pairing among $x_1', x_2', x_3', x_4'$), here are the different ways of connecting the two bridges into excursion(s):
\begin{itemize}
\item If  the two bridges are connected into one single excursion, then in the same way as the non-crossing case, there are $8$ ways (or $4$ modulo orientations) of doing it. For each  pattern in $A$, the product of the masses of the excursions in $A$ is equivalent to a constant factor  (depending on the pattern and on $X'$) times $\Pb(\Nc=1)$ times the density function $f$ applied to the corresponding pair. 

\item If the two bridges are connected into two different excursions, then the excursions in $A$ will contribute a constant factor  (depending on the pattern and on $X'$) times $\Pb(\Nc=2)$ times the density function $f_2$ applied to the corresponding two pairs.
\end{itemize}

\subsubsection*{Conclusion of the proof}
For each combination of patterns, the integral of the product of the masses is equivalent to a constant factor (depending on the pattern) times $\Pb(\Nc=1) f$ or $\Pb(\Nc=2) f_2$ (for $f$ or $f_2$ applied to the corresponding pair(s) depending also on the pattern) times  $\exp(-\pi (|I_1|+|I_2|+|I_3|+|I_4|)/\eps)$.
The probability of $\{\Nc=1\}$ and the function $f$ are known by Section~\ref{S3} and $\Pb(\Nc=2) f_2(a_1, a_3, a_2, a_4), \Pb(\Nc=2) f_2(a_1,a_4,a_3,a_2)$ are known because  the parings $(a_1, a_3), (a_2, a_4)$ and $(a_1, a_4), (a_3, a_2)$ are non-crossing. The only term for which we do not yet know that it is determined by the law of the trace of $\Ec$ is $\Pb(\Nc=2) f_2(a_1,a_2,a_3,a_4)$. However, since the total sum of all the terms is given by $\Pb(A^\eps(a_1,a_2,a_3,a_4))$ (which is determined by the law of the trace of $\Ec$), we can conclude that $\Pb(\Nc=2) f_2(a_1,a_2,a_3,a_4)$ is also determined by the law of the trace of $\Ec$.

We have  now determined $\Pb(\Nc=2) f_2 (a_1, a_2, a_3, a_4)$ for all distinct points $a_1, \ldots , a_4$. 
Since $f_2$ is a smooth function, this determines this function for all $a_1, \ldots, a_4$ in $(\partial U)^4$. 
Integrating this on $(\partial U)^4$, we obtain $\Pb(\Nc=2)$, and from this we can therefore also 
deduce $f_2$.

\section {Induction on the number of pairs} \label{S4}

Now let us proceed to the general induction. Again, we will provide details only for the parts of the arguments that involve new ideas (compared to how to deduce 
the result for two pairs from the result for one pair). 

Let us first define the density functions $f_k$ for all $k\ge 2$. If  $\Pb({\mathcal N}=k)>0$, then on the event $\{{\mathcal N}=k\}$, 
we define $( x_1,x_2) ,\ldots, (x_{2k-1}, x_{2k})$ to be an ordered  set of $k$ ordered pairs in $\Dc$ by assigning a uniformly chosen order between 
the $k$ pairs, and a uniformly chosen order for each of the pairs. We let $f_k$ be the density function on $(\partial U)^{2k}$ thus obtained.
Just as before, the function $f_k$ is
positive and smooth. 

Our inductive assumption is the following: We assume that for all $k\le n-1$ we know $\Pb({\mathcal N}=k)$ and the corresponding density function $f_k$. 
If $\sum_{k=1}^{n-1} \Pb({\mathcal N}=k)<1$, then we want to work out the value of $\Pb({\mathcal N}=n)$ and the function $f_n$.

\subsection*{Non-crossing pairs}
We first consider the case where the $n$ pairs $(a_1,a_2),\cdots, (a_{2n-1}, a_{2n})$ of distinct points are {\em non-crossing}, i.e. for any 
$1 \le i < j \le n$, the segments $[a_{2i-1}, a_{2i}]$ and $[a_{2j-1}, a_{2j}]$ are disjoint. This part of the proof will be very similar to the 
corresponding part when $n=2$.

The tubes $T_i=T^\eps(a_{2i-1}, a_{2i})$ for $i=1,\cdots, n$ are therefore also disjoint for $\eps$ small enough.
Let $A^\eps=A^\eps(a_1,\cdots, a_{2n})$ be the event that (i) $\Tc(\Ec)\cap U\subset\cup_{i=1}^{2n} T_i$ and that (ii) for all $i$,  $(\Tc(\Ec)\cap T_i)\cup [a_{2i-1}] \cup [a_{2i}]$ is connected and has a cut point disconnecting $ [a_{2i-1}]$ from $ [a_{2i}]$.
The event $A^\eps$ is then measurable with respect to the trace of $\Ec$.

\begin{figure}[h!]
        \centering
        \includegraphics[width=0.45\textwidth,page=1]{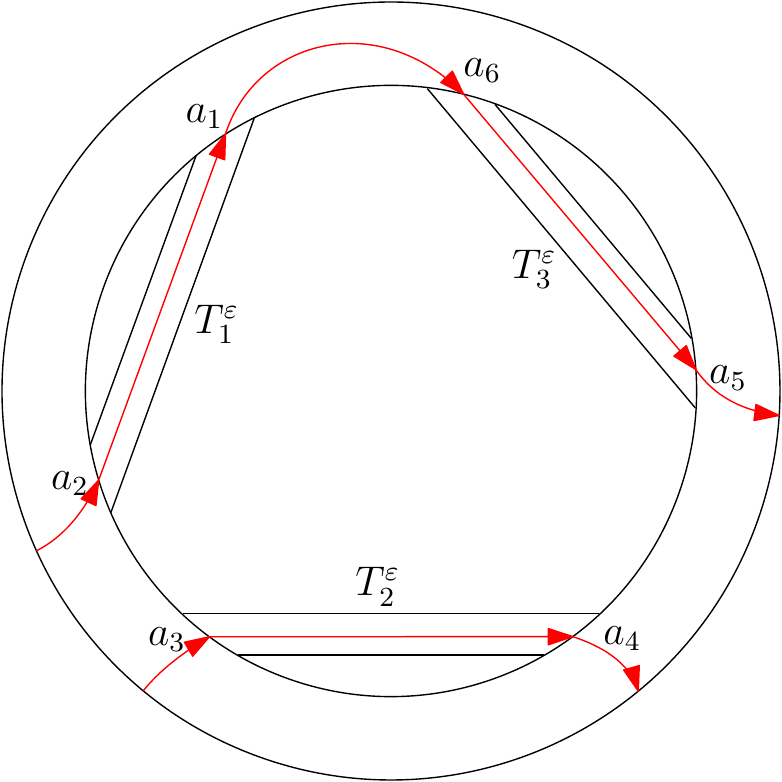}
        \caption{One way of connecting three bridges into two excursions.}
        \label{tubes}
\end{figure}

On $A^\eps$, the union of the excursions in $\Fc$ stays in $\cup_{i=1}^{2n} T_i \cup A$ and crosses each tube $T_i$ exactly once (for example see Figure~\ref{tubes}).
It is not difficult to see that $\Pb(A^\eps)$ is equivalent to $\sum_{k=1}^n \Pb(A^\eps, \Nc=k)$ and the term $\Pb(A^\eps, \Nc>n)$ is of smaller order.
We can compute $\Pb(A^\eps, \Nc=k)$ for each $k\le n-1$ by enumerating all possible ways of connecting the $n$ bridges in the $n$ tubes into $k$ excursions via excursions in $A$. By making similar decompositions as in the two pairs case, it is easy to show that  there exist some function $\tilde f_k$ depending only on $f_k$ (which is known by the induction assumption) such that
$$\Pb(A^\eps\mid \Nc=k) \underset{\eps\to 0}{\sim}  \tilde f_k(a_1,\cdots, a_{2n}) \times \eps^{2n} \times \exp(-\pi(|T_1|+\cdots + |T_n|)/\eps).$$
On the event $A^\eps\cap\{\Nc=n\}$, the $n$ bridges in the $n$ tubes must belong to $n$ excursions, which gives rise to
\begin{align*}
\Pb(A^\eps \mid \Nc=n) \underset{\eps\to 0}{\sim}  & 2^n |\Sc_n| \times f_n(a_1,\cdots, a_{2n}) \times c_0(a_1, a_2)\cdots c_0(a_{2n-1}, a_{2n})\\
 & \times \eps^{2n} \times   \exp(-\pi(|T_1|+\cdots + |T_n|)/\eps),
\end{align*}
where $\Sc_n$ is the set of permutations of $n$ elements.
This allows us to deduce $\Pb(\Nc=n) f_n(a_1,\cdots, a_{2n})$ for all set of 
$n$ non-crossing pairs.

\subsection*{Crossing pairs}\label{sec:crossing}
Let us note that just as in the case of two pairs (and we will briefly explain this in Section~\ref {Sfinal}), 
this case of crossing pairs could be bypassed if the only goal would be to establish Corollary~\ref {main corollary}. 
The proof will be reminiscent of the case of two crossing pairs, but we will use an additional induction over the ``number of crossings'' of the considered $n$ crossing pairs.

\subsubsection*{Setting up the induction over the number of crossings} 
We now consider a set of  $n$ pairs $(a_1,a_2),\cdots, (a_{2n-1}, a_{2n})$ of distinct points such that for any $1 \le i_1 < i_2 < i_3 \le n$, 
\begin{align}\label{concurrent}
[a_{2i_1-1}, a_{2i_1}] \cap  [a_{2i_2-1}, a_{2i_2}]\cap  [a_{2i_3-1}, a_{2i_3}] = \emptyset.
\end{align}
Our goal is to show that $\Pb(\Nc=n) f_n(a_1,\cdots, a_{2n})$ can be determined from the law of the trace of $\Ec$.
The set of $n$ pairs of distinct points satisfying (\ref{concurrent}) being dense in $(\partial U)^{2n}$, this will suffice to conclude, as we can integrate to recover $\Pb(\Nc=n)$ and then determine $f_n$. 

For any such $n$ pairs, we define its \emph{number of crossings} $m$ to be the number of pairs of segments (out of the $n(n-1)/2$ pairs) that do intersect.
The idea is (for each fixed $n$) to derive the result via an induction on the number $m$ of crossings (so there is a double induction here). 
We already know that the result holds for all $n$ pairs with no crossings (and that it holds for all sets of less than $n$ pairs). 

For $m\ge 1$, assume that $\Pb(\Nc=n) f_n(a_1,\cdots,a_{2n})$ 
for all $n$ pairs $(a_1,a_2),\cdots, (a_{2n-1}, a_{2n})$ that have no more than $m-1$ crossings is determined by the law of the trace of $\Ec$. 
The goal of the next paragraphs is to determine $\Pb(\Nc=n) f_n(a_1,\cdots,a_{2n})$ for all sets of $n$ pairs with $m$ crossings. 

\subsubsection*{Decomposition}
Let us fix $n$ pairs $(a_1,a_2),\cdots, (a_{2n-1}, a_{2n})$ with $m$ crossings.
For all $i$, let $T_i=T^\eps(a_{2i-1}, a_{2i})$. Condition~\eqref{concurrent} ensures that if $\eps$ is small enough, then for any $1\le i_1<i_2<i_3\le n$, one has $T_{i_1} \cap T_{i_2} \cap T_{i_3}=\emptyset$.
We call each of the components of  $T_i\setminus(\cup_{1\le k\le n,\, k\not=i} T_k)$ a \emph{sub-tube}, see Figure~\ref{sub-tube}. 
The boundary of each sub-tube is naturally divided into four parts where the two smaller parts are called its \emph{ends}. 
Let $A^\eps=A^\eps(a_1,\cdots, a_{2n})$ be the event that 
(i) $\Tc(\Ec)\cap U \subset \cup_{i=1}^{2n} T_i$ and that
(ii)  for all any sub-tube $S$ with two ends $[e_1]$ and $[e_2]$, $(\Tc(\Ec)\cap S)\cup[e_1]\cup[e_2]$ is connected and has a cut point separating $[e_1]$ and $[e_2]$.
The event $A^\eps$ is then measurable with respect to the trace of $\Ec$.

\begin{figure}[h]
\centering
  \includegraphics[width=.4\linewidth, page=2]{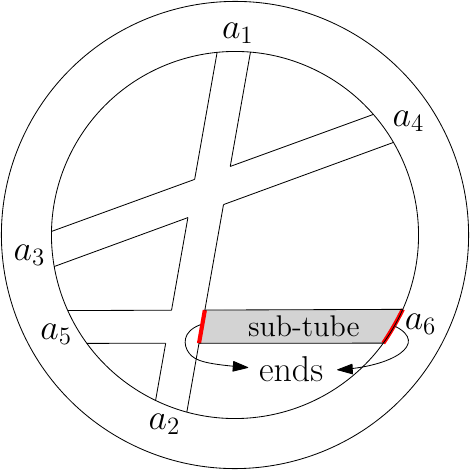}
  \caption{The tubes and sub-tubes. Here $n=3$ and $m=2$.}
  \label{sub-tube}
\end{figure}

As for the two pairs case, on the event $A^\eps$, we decompose the excursions in $\Fc$ into $n$ bridges with $2n$ ends $x_i'\in[a_i]$ for $i=1,\cdots, 2n$ 
(whose union crosses each sub-tube exactly once) and some excursions in $A$ connecting these bridges into excursions in $\Fc$. Conditionally on the family $X':= 
(x_i')_{i \le 2n}$ and on the pairing of these $2n$ points, the bridges and the excursions in $A$ are independent. We can therefore respectively enumerate all possible configurations for the bridges and for the excursions in $A$ and compute the  masses of the corresponding measures. 
Finally, for each combination of configurations, we will integrate the product of the corresponding masses and then sum up all the terms to get $\Pb(A^\eps)$.

\subsubsection*{Crossing-configurations and masses of the $n$ bridges} 
At each of the $m$ crossroads, there are three ways to connect the incoming pieces of excursions in the four adjacent sub-tubes (see Figure~\ref{fig:cross2}): 
two \emph{non-crossing patterns} and one \emph{crossing pattern}. We define a \emph{crossing-configuration} 
to be a function that assigns to each of the  $m$ crossroads one of the three patterns. For each of the $3^m$ crossing-configurations, the connection rule 
 gives rise to $n$ bridges joining the $2n$ points in $[a_1], \ldots, [a_{2n}]$ via some pairing, 
 and possibly also to closed loops, so that the union of the bridges and the loops 
 do cross each sub-tube exactly once. We call a  crossing-configuration \emph{admissible} (i.e., it represents a possible configuration of how the $n$ bridges cross the sub-tubes on the event $A^\eps$) if it contains no loop. 

So, each admissible crossing-configuration induces a pairing of the $2n$ points $(a_1, \ldots, a_{2n})$. 
An important simple observation is that the crossing-number of this new pairing can not be larger 
than the crossing-number $m$ of the initial pairing $(a_1, a_2), \ldots, (a_{2n-1}, a_{2n})$. 
Furthermore, the unique crossing-configuration that gives rise to the crossing number $m$ is the one where at each of the $m$ cross-roads, 
one assigns the crossing pattern (in other words, all the other crossing-configurations give rise to a pairing with a smaller crossing number). 

For each given admissible crossing-configuration, the product of the masses of the $n$ bridges is equivalent to a
constant (depending on $X'$ and on the crossing-configuration) 
times $\exp(-\pi L /\eps)$, where $L=\sum_{i=1}^n |a_{2i-1}-a_{2i}|$ is the total length of all tubes.

\subsubsection*{Connecting-configurations and masses of the excursions in $A$} Given the pairing of $X'$ (or equivalently of $(a_1, \ldots, a_{2n})$) that comes from the afore-mentioned crossing-configuration, 
we enumerate all possible ways of connecting these $n$ bridges into excursions in $\Fc$ by adding excursions in $A$. More precisely, each endpoint of a bridge is either connected to an endpoint of another bridge, or to the boundary $\partial\Ub$, and we call such a way of connection a \emph{connecting-configuration}. We emphasize that for any given pairing of $X'$, a connecting-configuration is determined by the excursions in $A$ only. 
This step is the same as the non-crossing case. For all $k< n$, if a connecting-configuration connects the $n$ bridges into $k$ excursions, 
then it will contribute a constant (depending soley on $X'$ and on the connecting-configuration) times $\Pb(\Nc=k)$ times $f_k$ 
(applied to the corresponding $k$ pairs in $\Dc$). If $\Nc=n$, then there the contribution is a constant factor times
$\Pb(\Nc=n)$ times $f_n$ (applied to the corresponding $n$ pairs).

\subsubsection*{Conclusion of the proof}
For each combination of admissible crossing-configurations and connecting-configurations, we integrate the product of the masses of the corresponding measures with respect to $\prod d\lambda(x_i')$ on $\prod [a_i]$.
Keeping in mind that the functions $f_k$ are smooth, we see that each resulting term 
is equivalent to some constant times $\Pb(\Nc=k)$ times $f_k$ (applied to the corresponding pairs) times 
$\eps^{2n} \exp(-\pi L/\eps)$, for $1\le k\le n$. By the induction assumption on $\Nc$, we know that 
only the terms involving $\Pb(\Nc=n) f_n$ are not (yet) determined by the law of the trace of $\Ec$.
However, we have argued that only one crossing-configuration gives rise to the pairing $(a_1,a_2), \cdots, (a_{2n-1}, a_{2n})$ with $m$ crossings, 
and that all the other crossing-configurations give rise to pairings with at most $m-1$ crossings. Hence, by the induction hypothesis on $m$, 
we see that by subtracting all the already known terms from $\Pb(A^\eps)$, we can also determine  $\Pb(\Nc=n) f_n(a_1,\cdots, a_{2n})$ from the law of the trace of $\Ec$, for this 
$(a_1, \ldots, a_{2n})$ with $m$ crossings. 

This then completes the induction over $m$, and shows that for any $2n$ distinct points $a_1, \ldots, a_n$ 
satisfying (\ref{concurrent}), the quantity  $\Pb(\Nc=n) f_n (a_1, \ldots , a_{2n})$  is determined by the law of the trace of $\Ec$. 
This finally allows to determine $\Pb(\Nc=n)$ and the function $f_n$. 

This completes the induction on $n$, and concludes the proof of  Lemma~\ref{main lemma} as well as of Proposition~\ref{main theorem}.

\section {Some comments on Brownian loop-soup cluster decompositions}       
\label {Sfinal} 
We now come back to features of the Brownian loop-soup clusters, in the set-up described in the second part of the introduction:  
Let $\Lambda$ be a Brownian loop-soup in $\Ub$ with intensity $c\le 1$. We call 
$\partial$ the outer boundary of the outermost cluster surrounding the origin, we let $O= O_\partial$ be the open domain surrounded by $\partial$ 
and we denote the collection of Brownian loops that stay in $\overline O$ by $\Lambda_0$. 
We also denote by $\phi_\partial$ the conformal map from $O$ onto $\Ub$ such that $\phi_\partial(0)=0$ and $\phi_\partial'(0)>0$. Since $\partial$ is a 
continuous simple loop, $\phi_\partial$ can be extended by continuity to a continuous one-to-one map from $\overline O$ onto the closed unit disk. 

Recall that it is proved in \cite{Qian-Werner} that $\phi_\partial(\Lambda_0)$ is independent of $\partial$ and it is the union of two independent sets of loops: (1) a Brownian loop-soup in $\Ub$ and (2) the set $\phi_\partial(\Lambda_\partial)$ where $\Lambda_\partial$ is the set of loops in $\Lambda_0$ that touch $\partial$. Furthermore, the law of $\phi_\partial(\Lambda_0)$  is conformally invariant. Note that each loop of $\phi_\partial (\Lambda_0)$ can be decomposed into excursions away from the unit circle.

The first goal of the present section is to prove the following lemma:
\begin{lemma}
\label {lemma6}
The family of excursions induced by $\phi_\partial(\Lambda_\partial)$ is a locally finite point process of Brownian excursions in $\Ub$.
\end{lemma}

Together with Proposition~\ref {main theorem} and the description of the law of the trace of $\Lambda_\partial$ for $c=1$ in \cite {Qian-Werner}, this implies Corollary~\ref {main corollary}. 

\begin{proof}
The local finiteness part of the statement is immediate, because there are only finitely many loops in $\Lambda_\partial$ that reach any given compact subset $K$ of $O$ (this is due to the local finiteness of the original loop-soup), and each loop in $\Lambda_\partial$ can create only finitely many excursions to $K$ (because each loop is a continuous loop).
It then only remains to prove that, conditionally on $\partial$ and on the pairs of extremities of the excursions on $\partial$, the excursions are distributed as independent Brownian excursions in $O$ with these given extremities.

Let  $L$, $L_1$ and $L_2$ denote three given concentric circles with radii $1>r>r_1>r_2$ around the origin in $\Ub$. We let $A$ denote 
the annulus between the unit circle and  $L$. We will be interested in the event that $\partial$ is contained in $A$, see Figure~\ref{loop-soup1}. 

There are almost surely 
finitely many loops in $\Lambda$ that intersect both $L_1$ and $L_2$ and every such loop makes almost surely a finite number of crossings between $L_1$ and $L_2$.
Everyone of these loops can be decomposed into the concatenation of a number of excursions outside of the disk $O_2$ (encircled by $L_2$) that do reach $L_1$ with bridges (in the disk $O_1$) that join the endpoints (on $L_2$) of these excursions in a certain order. 
We then know the following fact:
\smallskip

\noindent ($*$)
{\em If one conditions on the former part (i.e., on the family of excursions) and the pairing of the end-points of the bridges, then the remaining bridges are 
independent and distributed according to the bridge measure in $O_1$.}
\smallskip

\noindent In particular, ($*$) implies that when one resamples all these bridges, one gets another loop-soup $\tilde \Lambda$ (and we define $\tilde \partial$, the outer boundary of its outermost cluster that surrounds the origin).

Let us now suppose that the boundary $\partial$ is a subset of $A$. We now argue that almost surely, $\tilde \partial \subset A$, following similar 
ideas as in the proof of \cite[Lemma 4]{Qian-Werner}: We can note that all the points on $\partial$ will still be on the boundary of some macroscopic 
loop-soup cluster of $\tilde \Lambda$. Each of these loop-soup clusters has also to intersect $O_1$, which implies 
that there can only be finitely many of them (the clusters in a loop-soup are locally finite \cite{MR2979861}). 
But if there is more than one, and only finitely many of them, then it means that at least two of them are at zero distance from each other, which 
is known to be impossible (the clusters in a loop-soup are all disjoint \cite{MR2979861}). 
Therefore, we get indeed that on the event $\partial\subset A$, we necessarily have $\tilde \partial = \partial$, so that the event $\partial \subset A$ (and 
distribution of $\partial$ on this event) is 
preserved by the resampling operation in ($*$). 

\begin{figure}[h]
\centering
  \includegraphics[width=\linewidth]{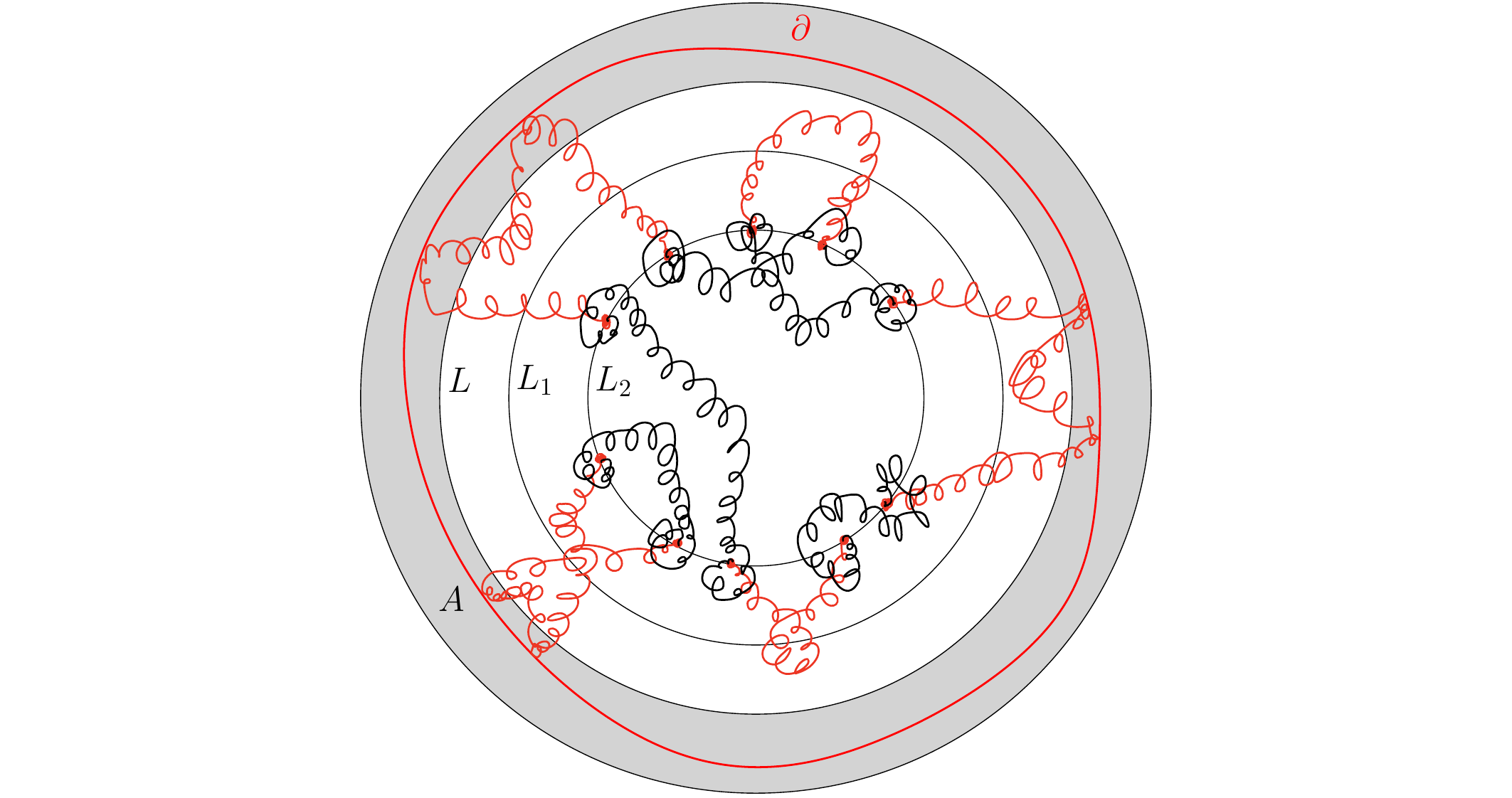}
  \caption{Decomposition of the loops in $\Lambda$ that intersect both $L_1$ and $L_2$ into excursions outside of $O_2$ that do reach $L_1$ (red) and bridges in $O_1$ (black)}
  \label{loop-soup1}
\end{figure}

We can now fix $r_1$ and $r_2$ and let $r$ tend to $1$. On the event $\partial\subset A$, the conformal map $\phi_\partial$  then tends to the identity map. 
Let $\hat L_1, \hat L_2$ and $\hat O_1, \hat O_2$ be the images of $L_1, L_2$ and $O_1, O_2$ under $\phi_\partial$.
Then the statement ($*$) (conditionally on $\partial\subset A$) can also be viewed as a statement on $\phi_\partial(\Lambda_0)$.
More precisely, we can decompose $\phi_\partial(\Lambda_0)$ into a number of excursions outside of $\hat O_2$ that do reach $\hat L_1$ with end-points on $\hat L_2$ and an equal number of bridges in $\hat O_1$ connecting those end-points. Conditionally on the excursions and on how their end-points should be paired by the bridges, the bridges are distributed as independent Brownian bridges in $\hat O_1$.
Note that $\phi_\partial(\Lambda_0)$ is in fact independent of $\partial$ (hence also of $\hat L_1$ and $\hat L_2$). Moreover, $\hat L_1$ and $\hat L_2$ tend to $L_1$ and $L_2$. This gives rise to the following statement in the limit:
\smallskip

\noindent
($**$) {\em We can decompose $\phi_\partial(\Lambda_0)$ into a number of excursions outside of $O_2$ that do reach $L_1$ with end-points on $L_2$ and an equal number of bridges in $O_1$ connecting those end-points. Conditionally on the excursions and on how their end-points should be paired by the bridges, the bridges are distributed as independent Brownian bridges in $O_1$.}
\smallskip

\noindent

Now, the idea is to let $r_1$ tend to $1$ in ($**$). 
Let $\Ec$ be the set of excursions away from the unit circle induced by $\phi_\partial(\Lambda_\partial)$. We can then define $\Ec_\delta, \Cc, \Cc_\delta$ just as in the previous sections (even that we do not yet know that $\Ec$ is a point process). 
Then, almost surely as $r_1$ tends to $1$, the pairs of end-points on $L_2$ induced by the decomposition in ($**$) will converge to $\Cc_{1-r_2}$ (the convergence is for finite sets of pairs of points). This implies that, if we decompose $\Ec_{1-r_2}$  into the excursions in the annulus between $L$ and $L_2$ and the bridges in $\Ub$, then conditionally on the excursions and on $\Cc_{1-r_2}$, the bridges are distributed like independent Brownian bridges in $\Ub$ with endpoints given by $\Cc_{1-r_2}$.
Since this is true for all $r_2$, it implies that $\Ec$ is in fact a point process of Brownian excursions. 
\end{proof}

Note that in the $c=0+$ limit, the lemma has the following interpretation: Conditionally on the outer boundary $\partial$ of a Brownian loop, the excursions away from $\partial$ is a (locally finite) point process of Brownian excursions. This can be derived directly from the definition of the Brownian loop measure as we need not worry about the disconnection of clusters anymore.

Another related example is when we look at a Brownian excursion in the upper half-plane from $0$ to $\infty$. Conditionally on its right (or left) boundary $\partial$, the excursions away from $\partial$ form again a point process of excursions. 

However, even for these two examples, we do not yet have a full description of 
the distribution of the traces of the point processes, as opposed to the critical loop-soup case. 
We do nevertheless know that all the point processes in Lemma~\ref {lemma6} do satisfy some conformal restriction property \cite{Qian-loop-soup}.
        
\medbreak 

A final observation is that in the case where $c=1$, one can use the additional resampling properties of the Brownian loop-soup 
from \cite {MR3618142} that are specific to that case. In the setting of the proof of Lemma~\ref {lemma6}, these resampling properties show 
that one knows the conditional law of the pairing (in order to form the bridges) between end-points of the excursions away from $L_2$, given these excursions (but not given the pairing of their end-points); indeed, the conditional probability of each pairing is proportional to the product of the Green's function of the corresponding bridges in $O_1$.
In the setting of the proof of Proposition~\ref {main theorem}, it follows that once one knows the value of $f_n (a_1, \ldots, a_{2n})$ for any $n$ non-crossing pairs $((a_1, a_2), \ldots, (a_{2n-1}, a_{2n}))$, one 
can deduce the value of $f_n$ for all $n$ pairs. So, if one only wants to prove only Corollary~\ref {main corollary}, one can in fact bypass the part of the proof about 
crossing configurations. 

\section*{Acknowledgements}
WQ acknowledges the support of an Early Postdoc.Mobility grant of the SNF and a JRF of Churchill College, Cambridge.
WW acknowledges the support of the SNF grant \# 175505, and is part of the NCCR Swissmap.

\bibliographystyle{plain}
\bibliography{cr}  
        
\end{document}